\documentclass[a4paper,12pt]{article}
\usepackage{times, url}
\usepackage{amsmath,amssymb,amsthm,amsfonts}

\newtheorem{theorem}{Theorem}
\newtheorem{lemma}{Lemma}
\begin{document}
\setcounter{page}{1}

\begin{center}
{\LARGE \bf  Consecutive square-free values of the form $\mathbf{[\alpha p], [\alpha p]+1}$}
\vspace{8mm}

{\large \bf S. I. Dimitrov}
\vspace{3mm}

Faculty of Applied Mathematics and Informatics, Technical University of Sofia \\
8, St.Kliment Ohridski Blvd. 1756 Sofia, BULGARIA \\
e-mail: \url{sdimitrov@tu-sofia.bg}
\vspace{2mm}
\end{center}
\vspace{10mm}

\noindent
{\bf Abstract:}
In this short paper we shall prove that there exist infinitely many consecutive
square-free numbers of the form $[\alpha p]$, $[\alpha p]+1$, where $p$ is prime
and $\alpha>0$ is irrational algebraic number.
We also establish an asymptotic formula for the number of such square-free pairs
when $p$ does not exceed given sufficiently large positive integer. 
\vspace{10mm}

\section{Notations}
\indent

Let $N$ be a sufficiently large positive integer.
The letter $p$  will always denote prime number.
By $\varepsilon$ we denote an arbitrary small positive number,
not necessarily the same in different occurrences. We denote by $\mu(n)$ the M\"{o}bius function and by $\tau(n)$
the number of positive divisors of $n$. As usual $[t]$ and $\{t\}$ denote the integer part, respectively, the
fractional part of $t$. Instead of $m\equiv n\,\pmod {k}$ we write for simplicity $m\equiv n\,(k)$.
Moreover $e(t)$=exp($2\pi it$). Let $\alpha>0$ be irrational algebraic number.
As usual $\pi(N)$ is the prime-counting function.

Denote
\begin{equation}\label{sigma}
\sigma=\prod\limits_{p}\left(1-\frac{2}{p^2}\right)\,.
\end{equation}

\section{Introduction and statement of the result}
\indent

In 1932 Carlitz \cite{Carlitz} proved that there exist infinitely many consecutive
square-free numbers. More precisely he established the asymptotic formula
\begin{equation}\label{Carlitz}
\sum\limits_{n\leq N}\mu^2(n)\mu^2(n+1)=\sigma N+\mathcal{O}\big(N^{\theta+\varepsilon}\big)\,,
\end{equation}
where $\sigma$ is denoted by \eqref{sigma} and $\theta=2/3$.

Subsequently the reminder term of \eqref{Carlitz} was improved by Mirsky \cite{Mirsky}
and Heath-Brown \cite{Heath-Brown}.
The best result up to now belongs to Reuss \cite{Reuss} with $\theta=(26+\sqrt{433})/81$.

In 2008  G\"{u}lo\u{g}lu and  Nevans \cite{Guloglu} showed by asymptotic formula that the sequence
\begin{equation}\label{sequence}
\{[\alpha n]\}_{n=1}^\infty
\end{equation}
contains infinitely many square-free numbers, where $\alpha>1$ is irrational number of finite type.

Recently Akbal \cite{Akbal} considered the sequence \eqref{sequence} with prime numbers
and proved the when $k\geq2$ and $\alpha>0$ is of type $\tau\geq1$, then there exist infinitely many
$k$-free numbers of the form $[\alpha p]$.
Akbal also established an asymptotic formula for the number of such $k$-free
numbers when $p$ does not exceed given sufficiently large real number $x$.

As a consequence of his result Akbal obtained the following
\begin{theorem}
Let $\alpha>0$ be an algebraic irrational number. Then
\begin{equation*}
\sum\limits_{p\leq N}\mu^2([\alpha p])
=\frac{6}{\pi^2}\pi(N)+\mathcal{O}\left(N^{\frac{9}{10}+\varepsilon}\right)\,.
\end{equation*}
\end{theorem}
\begin{proof}
See (\cite{Akbal} , Corollary 1).
\end{proof}

In 2018 the author \cite{Dimitrov1} showed that for any fixed $1<c<22/13$
there exist infinitely many consecutive square-free numbers of the form $[n^c], [n^c]+1$.

Recently the author \cite{Dimitrov2} proved that there exist infinitely many
consecutive square-free numbers of the form $x^2+y^2+1$, $x^2+y^2+2$.

Also recently the author \cite{Dimitrov3} showed that there exist infinitely many
consecutive square-free numbers of the form $[\alpha n]$, $[\alpha n]+1$,
where $\alpha>1$ is irrational number with bounded partial quotient
or irrational algebraic number.

Define
\begin{equation}\label{SigmaNalpha}
\Sigma(N, \alpha)=\sum\limits_{p\leq N}\mu^2([\alpha p]) \mu^2([\alpha p]+1)\,.
\end{equation}
Motivated by these results and following the method of Akbal \cite{Akbal}
we shall prove the following theorem.
\begin{theorem}\label{Mytheorem}
Let $\alpha>0$  be  irrational algebraic number.
Then for the sum $\Sigma(N, \alpha)$  defined by \eqref{SigmaNalpha} the asymptotic formula
\begin{equation}\label{asymptoticformula}
\Sigma(N, \alpha)=\sigma \pi(N)+\mathcal{O}\left(N^{\frac{9}{10}+\varepsilon}\right)
\end{equation}
holds. Here $\sigma$ is defined by \eqref{sigma}.
\end{theorem}
From Theorem \ref{Mytheorem} it follows that there exist infinitely many
consecutive square-free numbers of the form $[\alpha p]$, $[\alpha p]+1$,
where $p$ is prime and $\alpha>0$ is irrational algebraic number.

\newpage

\section{Lemmas}
\indent

\begin{lemma}\label{Erdos}\textbf{ (Erd\"{o}s-Tur\'{a}n inequality)}
Let  $\{t_k\}_{k=1}^K$  be a sequence of real numbers. Suppose that
$\mathcal{I}\subset[0,1)$ is an interval. Then
\begin{equation*}
\Big| \#\{ k\leq K\,:\,   \{t_k\}\in  \mathcal{I} \} - K|\mathcal{I}| \Big|
\ll \frac{K}{H}+\sum\limits_{h\leq H}\frac{1}{h}\left|\sum\limits_{k\leq K}e(ht_k)\right|
\end{equation*}
for any $H \gg1$. The constant in the $\mathcal{O}$-term is absolute.
\end{lemma}
\begin{proof}
See (\cite{Baker}, Theorem 2.1).
\end{proof}

\begin{lemma}\label{Akbal-Lemma}
Suppose that $H, D, T, N \geq 1$.  Let $\alpha>0$ be  irrational algebraic number.
Then
\begin{align}\label{Dimitrov-est}
\sum\limits_{H<h\leq 2H}\sum\limits_{D<d\leq 2D}\sum\limits_{T<t\leq 2T}
\left|\sum\limits_{p\leq N}e\left(\frac{\alpha hp}{d^2t^2}\right)\right|
&\ll(H DT N)^\varepsilon \Big(H^{1/2} D^2 T^2 N^{1/2}+ H^{3/5} DT N^{4/5}\nonumber\\
&+H DT N^{3/4} + H^{3/4} D^{3/2} T^{3/2} N^{3/4} \Big)\,.
\end{align}
\end{lemma}
\begin{proof}
This lemma is very similar to result of Akbal \cite{Akbal}.
Inspecting the arguments presented in (\cite{Akbal}, Lemma 3),
the reader will easily see that the proof of Lemma \ref{Akbal-Lemma}
can be obtained by the same manner.
\end{proof}

\section{Proof of the Theorem}
\indent

Assume
\begin{equation}\label{z}
2\leq z \leq(\alpha N)^{2/3}\,.
\end{equation}
We use \ \eqref{SigmaNalpha} and the well-known identity
$\mu^2(n)=\sum_{d^2|n}\mu(d)$ to write
\begin{align}\label{SNalphaest1}
\Sigma(N, \alpha)&=\sum\limits_{p\leq N}\mu^2([\alpha p])\mu^2([\alpha p]+1)
=\sum\limits_{p\leq N}\sum\limits_{d^2|[\alpha p]}\mu(d)
\sum\limits_{t^2|[\alpha p]+1}\mu(t)\nonumber\\
&=\sum\limits_{d, t \atop{(d,t)=1}}\mu(d)\mu(t)
\sum\limits_{p\leq N\atop{[\alpha p]\equiv0\,(d^2)\atop{[\alpha p]+1\equiv0\,(t^2)}}}1
=\Sigma_1(N)+\Sigma_2(N)\,,
\end{align}
where
\begin{equation}\label{SigmaN1}
\Sigma_1(N)=\sum\limits_{dt\leq z\atop{(d,t)=1}}\mu(d)\mu(t)
\sum\limits_{p\leq N\atop{[\alpha p]\equiv0\,(d^2)
\atop{[\alpha p]+1\equiv0\,(t^2)}}}1\,,
\end{equation}
\begin{equation}\label{SigmaN2}
\Sigma_2(N)=\sum\limits_{dt>z\atop{(d,t)=1}}\mu(d)\mu(t)
\sum\limits_{p\leq N\atop{[\alpha p]\equiv0\,(d^2)
\atop{[\alpha p]+1\equiv0\,(t^2)}}}1\,.
\end{equation}

\newpage

\textbf{Estimation of} $\mathbf{\Sigma_1(N)}$

From \eqref{SigmaN1} and Chinese remainder theorem we obtain
\begin{equation}\label{SigmaN1est1}
\Sigma_1(N)=\sum\limits_{dt\leq z\atop{(d,t)=1}}\mu(d)\mu(t)
\sum\limits_{p\leq N\atop{[\alpha p]\equiv q\,(d^2t^2)}}1\,,
\end{equation}
where $1\leq q\leq d^2t^2-1$.\\
It is easy to see that the congruence $[\alpha p]\equiv q\,(d^2t^2)$ is tantamount to
\begin{equation}\label{tantamount}
\frac{q}{d^2t^2}<\left\{\frac{\alpha p}{d^2t^2}\right\}<\frac{q+1}{d^2t^2}\,.
\end{equation}
Bearing in mind  \eqref{SigmaN1est1}, \eqref{tantamount} and Lemma \ref{Erdos} we get
\begin{align}\label{SigmaN1est2}
\Sigma_1(N)&=\pi(N)\sum\limits_{dt\leq z\atop{(d,t)=1}}\frac{\mu(d)\mu(t)}{d^2t^2}
+\mathcal{O}\left(\frac{N}{H}\sum\limits_{dt\leq z}1\right)
+\mathcal{O}\left(\sum\limits_{dt\leq z\atop{(d,t)=1}}
\sum\limits_{h\leq H}\frac{1}{h}
\left|\sum\limits_{p\leq N}e\left(\frac{\alpha hp}{d^2t^2}\right)\right|\right)\nonumber\\
=&\pi(N)\left(\sum\limits_{d,t=1\atop{(d,t)=1}}\frac{\mu(d)\mu(t)}{d^2t^2}
-\sum\limits_{dt>z\atop{(d,t)=1}}\frac{\mu(d)\mu(t)}{d^2t^2}\right)
+\mathcal{O}\left(\frac{N}{H}\sum\limits_{dt\leq z}1\right)\nonumber\\
&+\mathcal{O}\left(\sum\limits_{dt\leq z\atop{(d,t)=1}}
\sum\limits_{h\leq H}\frac{1}{h}
\left|\sum\limits_{p\leq N}e\left(\frac{\alpha hp}{d^2t^2}\right)\right|\right)\,.
\end{align}
It is well-known that
\begin{equation}\label{Sum-Product}
\sum\limits_{d,t=1\atop{(d,t)=1}}\frac{\mu(d)\mu(t)}{d^2t^2}
=\prod\limits_{p}\left(1-\frac{2}{p^2}\right)\,.
\end{equation}
On the other hand
\begin{equation}\label{Sumdtz1}
\sum\limits_{dt>z\atop{(d,t)=1}}\frac{\mu(d)\mu(t)}{d^2t^2}\ll\sum\limits_{dt>z}\frac{1}{d^2t^2}
=\sum\limits_{n>z}\frac{\tau(n)}{n^2}
\ll\sum\limits_{n>z}\frac{1}{n^{2-\varepsilon}}\ll z^{\varepsilon-1}\,.
\end{equation}
By the same way
\begin{equation}\label{Sumdtz2}
\sum\limits_{dt\leq z}1=\sum\limits_{n\leq z}\tau(n)\ll z^{1+\varepsilon}\,.
\end{equation}
From \eqref{SigmaN1est2} --  \eqref{Sumdtz2} it follows
\begin{equation}\label{SigmaN1est3}
\Sigma_1(N)=\sigma\pi(N)+\mathcal{O}\Big(\pi(N)z^{\varepsilon-1}\Big)
+\mathcal{O}\left(\frac{N}{H}z^{1+\varepsilon}\right)
+\mathcal{O}\left(\sum\limits_{dt\leq z}\sum\limits_{h\leq H}\frac{1}{h}
\left|\sum\limits_{p\leq N}e\left(\frac{\alpha hp}{d^2t^2}\right)\right|\right)\,,
\end{equation}
where $\sigma$ is denoted by \eqref{sigma}.\\
Splitting the range of $h$, $d$ and $t$ of the exponential sum in \eqref{SigmaN1est3}
into dyadic subintervals of the form $H<h\leq 2H$, $D<d\leq 2D$, $T<t\leq 2T$,
where $DT<z$ and applying Lemma \ref{Akbal-Lemma} we find
\begin{align}\label{Expsumest}
\sum\limits_{dt\leq z}\sum\limits_{h\leq H}\frac{1}{h}
\left|\sum\limits_{p\leq N}e\left(\frac{\alpha hp}{d^2t^2}\right)\right|
&\ll(H DT N)^\varepsilon \Big(D^2 T^2 N^{1/2}+  DT N^{4/5}+D^{3/2} T^{3/2} N^{3/4} \Big)\nonumber\\
&\ll(H z N)^\varepsilon \Big(z^2 N^{1/2}+  z N^{4/5}+z^{3/2} N^{3/4} \Big)\,.
\end{align}
Taking into account \eqref{z}, \eqref{SigmaN1est3}, \eqref{Expsumest} and choosing $H=N^{1/5}$ we obtain
\begin{equation}\label{SigmaN1est4}
\Sigma_1(N)=\sigma\pi(N)
+\mathcal{O}\bigg(N^\varepsilon \Big(z^2 N^{1/2}+  z N^{4/5}+z^{3/2} N^{3/4}+Nz^{-1} \Big) \bigg)\,.
\end{equation}
\textbf{Estimation of} $\mathbf{\Sigma_2(N)}$

By \eqref{z}, \eqref{SigmaN2}, \eqref{Sumdtz1} and Chinese remainder theorem we get
\begin{align}\label{SigmaN2est1}
\Sigma_2(N)&\ll\sum\limits_{dt>z}
\sum\limits_{n\leq N\atop{[\alpha n]\equiv0\,(d^2)
\atop{[\alpha n]+1\equiv0\,(t^2)}}}1
=\sum\limits_{dt>z}\sum\limits_{n\leq N\atop{[\alpha n]\equiv l\,(d^2t^2)}}1
\ll\sum\limits_{dt>z}\sum\limits_{m\leq [\alpha N]\atop{m\equiv l\,(d^2t^2)}}1\nonumber\\
&\ll N\sum\limits_{dt>z}\frac{1}{d^2t^2}\ll N^{1+\varepsilon}z^{-1}\,.
\end{align}
\textbf{The end of the proof}

Bearing in mind \eqref{SNalphaest1}, \eqref{SigmaN1est4}, \eqref{SigmaN2est1}
and choosing $z=N^{1/10}$ we establish the asymptotic formula \eqref{asymptoticformula}.

The theorem is proved.


\begin{thebibliography}{}

\bibitem{Akbal} Y. Akbal, {\it  A short note on some arithmetical properties
of the integer part of $\alpha  p$},
Turkish Journal of Mathematics, \textbf{43}(3), (2019), 1253 -- 1262.

\bibitem{Baker}  R. C. Baker , {\it Diophantine Inequalities (London Mathematical Society Monographs)},
New York, NY, USA: Clarendon Press, (1986).

\bibitem{Carlitz}L. Carlitz, {\it On a problem in additive arithmetic II},
Quart. J. Math., {\bf3}, (1932), 273 -- 290.

\bibitem{Dimitrov1}S. I.  Dimitrov, {\it  Consecutive square-free numbers of the form $[n^c], [n^c]+1$},
JP Journal of Algebra, Number Theory and Applications, \textbf{40}, 6, (2018), 945 -- 956.

\bibitem{Dimitrov2} S. I. Dimitrov, {\it On the number of pairs of positive integers
$x, y \leq H$ such that $x^2+y^2+1, x^2+y^2+2$ are square-free},
arXiv:1901.04838  [math.NT]  5 Jan 2019.

\bibitem{Dimitrov3} S. I. Dimitrov, {\it On the distribution of consecutive square-free
numbers of the form  $\mathbf{[\alpha n], [\alpha n]+1}$},
arXiv:1903.04545v2 [math.NT] 22 Mar 2019.

\bibitem{Guloglu} A. M. G\"{u}lo\u{g}lu, C. W. Nevans, {\it Sums of multiplicative functions over a Beatty sequence},
Bull. Austral. Math. Soc., \textbf{78}, (2008), 327 -- 334.

\bibitem{Heath-Brown}D. R. Heath-Brown, {\it The Square-Sieve and Consecutive Square-Free Numbers},
Math. Ann., {\bf266}, (1984), 251 -- 259.

\bibitem{Mirsky}L. Mirsky, {\it On the frequency of pairs of square-free numbers with a given difference},
Bull. Amer. Math. Soc., {\bf55}, (1949), 936 -- 939.

\bibitem{Reuss}T. Reuss, {\it Pairs of k-free Numbers, consecutive square-full Numbers},
arXiv:1212.3150v2  [math.NT]  19 Mar 2014.

\end{thebibliography}
\end{document}